\documentclass{article}

     \usepackage[nonatbib,preprint]{neurips_2021}

\usepackage[utf8]{inputenc} 
\usepackage[T1]{fontenc}    
\usepackage[colorlinks]{hyperref}       
\usepackage{url}            
\usepackage{booktabs}       
\usepackage{amsfonts}       
\usepackage{nicefrac}       
\usepackage{microtype}      
\usepackage{mathtools,mathrsfs}
\usepackage{amsthm,Theorems}
\usepackage{xcolor}
\usepackage[font={small}]{caption}
\usepackage{wrapfig}
\usepackage[inline]{enumitem}
\title{Theoretical Foundations for the Dynamic Mode Decomposition of High Order Dynamical Systems}

\author{%
  Joel A. Rosenfeld\thanks{The following YouTube playlist discusses the content of this manuscript: \url{https://youtube.com/playlist?list=PLldiDnQu2phtB0fZbr-wOgisJS_684bNP}}\\
  Department of Mathematics and Statistics\\
  University of South Florida\\
  Tampa, Fl 33620 \\
  \texttt{rosenfeldj@usf.edu} \\
   \And
   Benjamin P. Russo \\
   Department of Mathematics\\
   Farmingdale State College\\
   Farmingdale, NY 11735 USA \\
   \texttt{russobp@farmingdale.edu} \\
   \And
   Rushikesh Kamalapurkar \\
   School of Mechanical and Aerospace Engineering \\
   Oklahoma State University\\
   Stillwater, OK 74078 \\
   \texttt{rushikesh.kamalapurkar@okstate.edu} \\
}

\DeclareMathOperator{\vspan}{span}

\begin{document}

\maketitle

\begin{abstract}
  Conventionally, data driven identification and control problems for higher order dynamical systems are solved by augmenting the system state by the derivatives of the output to formulate first order dynamical systems in higher dimensions. However, solution of the augmented problem typically requires knowledge of the full augmented state, which requires numerical differentiation of the original output, frequently resulting in noisy signals. This manuscript develops the theory necessary for a direct analysis of higher order dynamical systems using higher order Liouville operators. Fundamental to this theoretical development is the introduction of signal valued RKHSs and new operators posed over these spaces. Ultimately, it is observed that despite the added abstractions, the necessary computations are remarkably similar to that of first order DMD methods using occupation kernels.
\end{abstract}

\section{Introduction}

Data driven methods for dynamical systems have developed significantly over the past 20 years (cf. \cite{budivsic2012applied,kutz2016dynamic,proctor2016dynamic,mauroy2020linear,mauroy2016global,brunton2016discovering}). Principle among them are those that leverage Koopman operators (also known as composition operators) over Hilbert function spaces to give a representation of finite dimensional discrete time dynamics as an operator over an infinite dimensional Hilbert space \cite{budivsic2012applied,williams2015kernel}. When a continuous time dynamical system is forward complete, it may be discretized by fixing a time-step, $\Delta t > 0$, to yield a discrete time system. In this setting, the Koopman operator has been demonstrated as an effective tool for extracting the underlying governing principles of a dynamical system, and for providing a model for the state which performs well over short time horizons via Dynamic Mode Decompositions (DMD). The collection of Koopman operators, indexed by $\Delta t$, is called the Koopman semigroup, and as $\Delta t \to 0$, the Koopman operators converge strongly to the \emph{Koopman generator}, given as $A_f g = \nabla g(\cdot) f(\cdot)$ \cite{budivsic2012applied}. However, there are many dynamical systems of interest that are not forward complete (e.g., models of the form $\dot x = 1 + x^2$, often encountered in mass-action kinetics in thermodynamics \cite[Section 6.3]{haddad2019dynamical}, chemical reactions \cite[Section 8.4]{toth2018reaction}, and species populations \cite[Section 4.2]{hallam2012mathematical}), and consequently, the breadth of applicability of Koopman based methods is limited. Koopman generators are a proper subset of Liouville operators, which are formally the same operators, but where the dynamics are not shackled by the requirement of forward completeness.

In \cite{rosenfeld2019occupation}, the concept of occupation kernels was introduced as functions inside of a RKHS that, given a signal $\theta:[0,T]\to \mathbb{R}^n$, represent the functional $g \mapsto \int_0^T g(\theta(t))dt.$ Occupation kernels generalize the concept of an occupation measure (cf. \cite{lasserre2008nonlinear}) by changing the setting from a collection of measures to that of a Hilbert space. Thus, occupation kernels can be leveraged as a basis in a Hilbert space for function approximation and projections (cf. \cite{rosenfeld2019occupation,rosenfeldCDC2019,rosenfeldDMD,rosenfeld2021dynamic,li2021fractional,russo2021motion}). Significant to the problem at hand, if $A_f$ is densely defined over a RKHS over $\mathbb{R}^n$ of continuously differentiable functions, and $\gamma:[0,T] \to \mathbb{R}^n$ is a trajectory satisfying $\dot \gamma = f(\gamma)$, then $\langle A_f g, \Gamma_{\gamma} \rangle_H = g(\gamma(T)) - g(\gamma(0)) = \langle g, K_{\gamma(T)} - K_{\gamma(0)}\rangle_H.$ Hence, $A_f^* \Gamma_{\gamma} = K_{\gamma(T)} - K_{\gamma(0)}$, where $K_x := K(\cdot,x)$, and $ K:\mathbb{R}^n\times\mathbb{R}^n\to\mathbb{R} $ denotes the kernel function of the RKHS. This relationship between $A_f$ and $K$ allows for the construction of a finite rank representation of $A_f$ using the trajectories of the system as the central unit of information.

The advantage gained by taking the occupation kernel perspective is two fold. The dynamics are now free from the requirements of forward completeness, and instead, the dynamics are required to provide a densely defined Liouville operator over a given RKHS. Different selections of RKHSs yield new classes of dynamics that meet these requirements, which adds considerably to the flexibility of the method. The other advantage is that by the nature of occupation kernels as representatives of integration, they take the form of $\Gamma_{\theta}(x) = \int_0^T K(x,\theta(t)) dt.$ Thus, $\Gamma_{\theta}$ interprets the trajectory data through an integral and is robust to sensor noise \cite{rosenfeld2019occupation}.

One limitation still present in the theory of data driven methods for dynamical systems is that of high order dynamics. Conventionally in systems theory, higher order dynamics are converted to first order systems of augmented state variables. For example, $\ddot x = f(x)$ can be adjusted to $z := \begin{pmatrix} x & \dot x\end{pmatrix}^T$ with $\dot z = \begin{pmatrix} z_2 & f(z_1)\end{pmatrix}^T.$ Theoretically, the augmentation is well justified, but it is computationally problematic in data driven methods. To estimate the new state variable of $z$, data driven methods must compute an approximation of the first derivative of $x$. Numerical derivatives can be noisy, and if the order of the system exceeds $2$, they are unreliable. For example, in \cite{brunton2016discovering}, where numerical differentiation is used for parameter fitting, a considerable amount of filtering was required to get good results. The sensitivity of numerical differentiation techniques to noise motivates the development of methods that avoid numerical differentiation altogether.

This manuscript introduces the necessary theoretical components for the development of a DMD routine for second order dynamical systems that avoids the use of numerical derivatives and augmented state variables. The exposition will be focused on second order dynamical systems. However, the developed methods may be readily adapted to higher order dynamical systems. Underlying the subsequent development are vector valued Reproducing Kernel Hilbert spaces (vvRKHSs), for which the relevant theory is presented in Section \ref{sec:vvRKHS}. Using vvRKHSs as a tool, Section \ref{sec:vectorvaluedframework} develops a signal valued RKHS, which is a Hilbert space of functions that map $d$ times continuously differentiable signals to a scalar valued RKHS over $[0,T]$. The signal valued RKHS framework allows for the formulation of well defined second order Liouville operators over the Hilbert space beyond the formal expression given in Section \ref{sec:hoLiouville}. Once the essential elements are established, Section \ref{sec:dmdsecondorder} presents a DMD method for the modeling of a second order dynamical system, which avoids the use of numerical derivatives.

\section{Vector Valued Reproducing Kernel Hilbert Spaces}\label{sec:vvRKHS}

This section presents the concept of vector valued RKHSs, which recently came to prominence with \cite{SCC.Carmeli.DeVito.ea2010}, though their roots extend further back (e.g. \cite{pedrick1957theory}). In the context of this manuscript the Hilbert space $\mathcal{Y}$ will be a scalar valued RKHS, which will facilitate the description of a function space on signals.

\begin{definition}Given a Hilbert space $\mathcal{Y}$ and a set $X$, a vector valued reproducing kernel Hilbert space, $H$, is a Hilbert space of functions mapping $X$ to $\mathcal{Y}$, where for each $x \in X$ and $v \in \mathcal{Y}$ the functional $g \mapsto \langle g(x), v \rangle_{\mathcal{Y}}$ is bounded.
\end{definition}

The Riesz representation theorem guarantees for each $x \in X$ and $v \in \mathcal{Y}$ the existence of a function $K_{x,v} \in H$ such that $\langle g, K_{x,v} \rangle_{H} = \langle g(x), v \rangle_{\mathcal{Y}}$ for all $g \in H$. It is readily apparent that the map $K_{x}:\mathcal{Y}\to H$, that maps $v$ to $K_{x,v}$, is linear, and as such, is expressed as $K_{x}v := K_{x,v}$.


\section{Problem Statement}\label{sec:problemstatement}

Given a collection of trajectories, $\{ \gamma_{i}:[0,T]\to \mathbb{R}^n \}_{i=1}^M$, corresponding to a second order dynamical system $\ddot \gamma = f(\gamma)$, where $f$ is unknown, we want to determine a model for a trajectory starting at $x(0)=x_0^0$ and $\dot x(0) = x_0^1$. In the following, the model is constructed from a finite rank representation of a second order Liouville operator, $B_f$, obtained via adjoint relations between the $B_f$, and a collection of occupation kernels.

\section{Signal valued RKHSs}\label{sec:vectorvaluedframework}


The objective of this section is to provide a definition of \emph{signal valued} RKHSs. These Hilbert spaces consist of functions that map from signals in $C^{d}\left([0,T],\mathbb{R}^n\right)$ to signals in $\mathcal{Y}$, which itself is a scalar valued RKHS over $[0,T]$. In contrast to scalar valued RKHSs which consist of functions acting on a state space, signal valued RKHSs treat a signal from $C^{d}\left([0,T],\mathbb{R}^n\right)$ as the fundamental unit of information. This allows for the treatment of operators acting on functions of signals, which include higher order Liouville operators. \begin{wrapfigure}[13]{r}{0.5\textwidth}
    \centering
    \includegraphics[width = 0.48\textwidth]{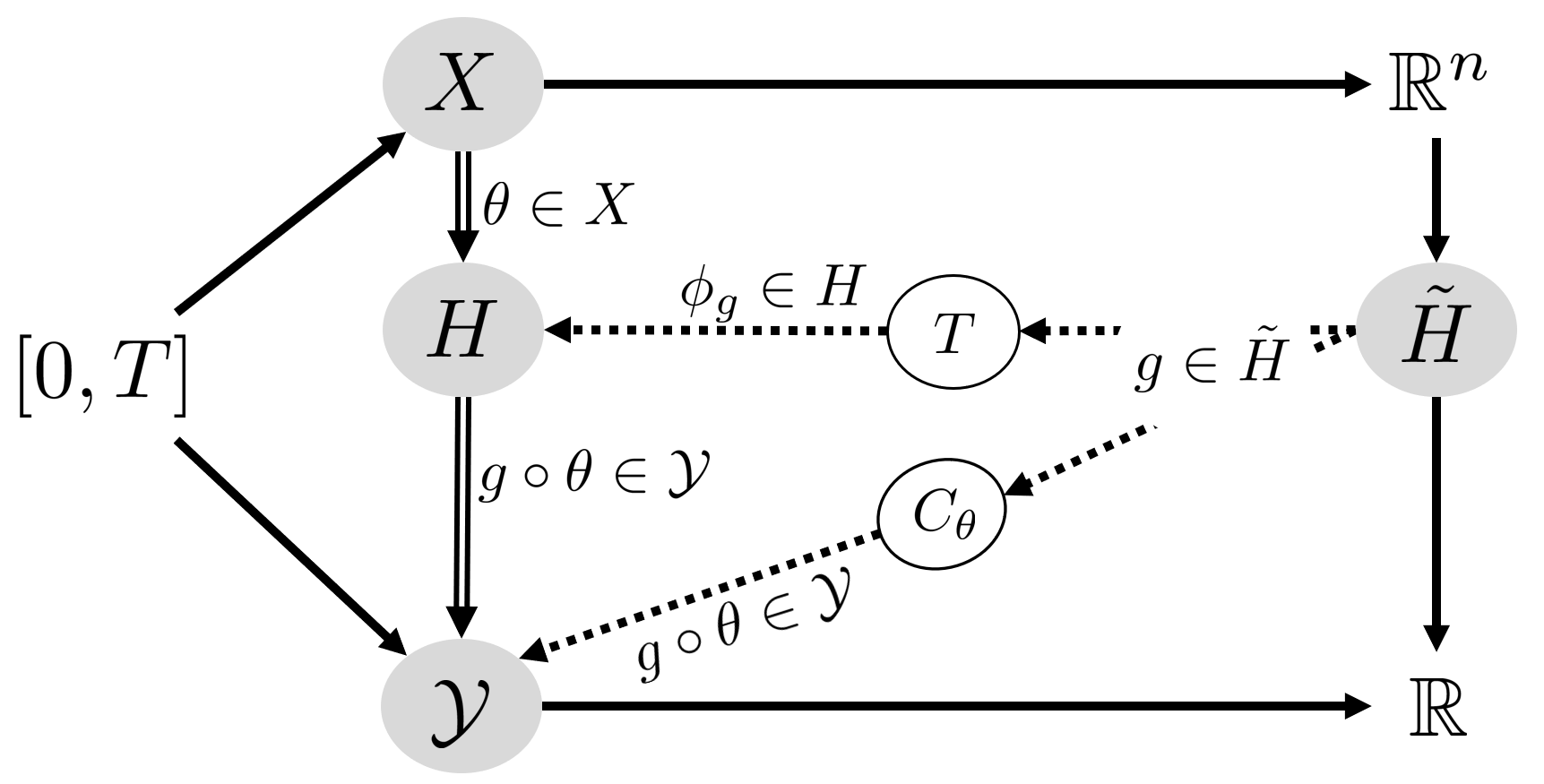}
    \caption{A visualization of relationships between vector spaces and operators defined in Theorem \ref{thm:construction}.}
    \label{fig:FunctionSpaces}
\end{wrapfigure}
Signal valued RKHSs will be derived as a special case of vector valued RKHSs, and arise from a mapping on scalar valued RKHSs.

In the following, three different RKHSs are under consideration (see Figure \ref{fig:FunctionSpaces}). The range space, $\mathcal{Y}$, is selected to be a scalar valued RKHS over $[0,T]$, with kernel function $\mathscr{K}$. To construct a vvRKHS of functions that map signals from $C^{d}\left([0,T],\mathbb{R}^n\right)$ (or a suitable substitute) to $\mathcal{Y}$, we define an auxiliary scalar valued RKHS, $\tilde H$, consisting of twice continuously differentiable functions from $\mathbb{R}^n$ to $\mathbb{R}$. For each $g \in \tilde H$, a map from $C^{d}\left([0,T],\mathbb{R}^n\right)$ to $\mathcal{Y}$ is obtained as $\phi_g[\theta](t) := g(\theta(t))$ for all $\theta \in C^{d}\left([0,T],\mathbb{R}^n\right)$ and $t \in [0,T]$. Theorem \ref{thm:construction} shows that the space of all such maps is a vvRKHS.

\begin{theorem}\label{thm:construction} Let $X = C^{d}\left([0,T],\mathbb{R}^n\right)$ for some $d \in \mathbb{N}$, and let $\tilde{H}$ be a scalar valued RKHS over $\mathbb{R}^n$. Moreover, suppose there exists a RKHS $\mathcal{Y}$ over $[0,T]$ where the composition operator $C_{\theta}: \tilde{H}\rightarrow \mathcal{Y}$ is a bounded operator for all symbols $\theta\in X$. Define the vector space 
$H(X):=\{\phi_g \ :\  g\in \tilde{H}\}$
of mappings $\phi_g:X\rightarrow \mathcal{Y}$ given by $\phi_g[\theta]:=g(\theta(\cdot))$, together with the inner product induced by $\tilde{H}$, $\langle \phi_{g_1},\phi_{g_2}\rangle_H=\langle g_1, g_2\rangle_{\tilde{H}}.$
Then $H(X)$ is a vvRKHS. 
\end{theorem}

\begin{proof}
Let $\mathcal{T}$ be the mapping from $\tilde H$ to $H$ given as $\mathcal{T}g = \phi_g$. It is immediate that $\mathcal{T}$ is an injection into $H$. Indeed if $g_1 \neq g_2$ for $g_1$ and $g_2$ in $\tilde H$, then there is a point $\omega \in \mathbb{R}^n$ such that $g_1(\omega) \neq g_2(\omega)$. Letting $\theta_\omega \in C^{d}\left([0,T],\mathbb{R}^n\right)$ denote the constant signal $\theta_\omega(t) \equiv \omega$, then $\phi_{g_1}(\theta_\omega) \neq \phi_{g_2}(\theta_\omega)$. {\color{black}Selecting the inner product on $H$ to be $\langle \phi_{g_1}, \phi_{g_2} \rangle_H = \langle g_1, g_2 \rangle_{\tilde H}$, $H$ can be seen to be complete with respect to the induced norm. Since $\tilde{H}$ is a Hilbert space, it is complete by definition. Hence, a Cauchy sequence in $H$, say $\{\phi_{g_n}\}$ gives rise to a Cauchy sequence in $\tilde{H}$, namely $\{g_n\}$. Since $\tilde{H}$ is complete, the sequence converges to a $g\in \tilde{H}$ and thus $\phi_{g_n}$ converges to $\phi_g$. Since $H=\mathcal{T}(\tilde{H})$, $\mathcal{T}$ is surjective.} What is left to resolve is to demonstrate that for any $\theta \in C^{d}\left([0,T],\mathbb{R}^n\right)$ and $v \in \mathcal{Y}$, the functional $\psi \mapsto \langle \psi[\theta], v \rangle_{\mathcal{Y}}$ is bounded. Note that there is a $g \in \tilde H$ such that $\mathcal{T}g = \psi$ and $\|\psi\|_H = \|g\|_{\tilde H}$. Now consider $|\langle \psi[\theta], v \rangle_{\mathcal{Y}}| = |\langle g \circ \theta, v \rangle_{\mathcal{Y}}| \le \| g \circ \theta \|_{\mathcal{Y}} \| v \|_{\mathcal{Y}} \le {\color{black}\|C_\theta\| \|g\|_{\tilde H} \| v\|_{\mathcal{Y}} = \|C_\theta\| \|\psi\|_{H} \| v\|_{\mathcal{Y}}.}$
\end{proof}

\begin{definition}
The vvRKHS, $H$, given in Theorem \ref{thm:construction} will be called the signal valued RKHS from $C^{d}\left([0,T],\mathbb{R}^n\right)$ to $\mathcal{Y}$ derived from $\tilde H$, more succinctly a signal valued RKHS, when the other quantities are understood from context.
\end{definition}

{\color{black}While a general characterization of pairs of RKHSs $(\tilde{H}, \mathcal{Y})$ that admit bounded composition operators $C_\theta$ is difficult, the following example analyzes one such pair.}

\begin{example}\label{ex:GaussianSobolev}
A possible choice for $\tilde H$ would be the native space of the Gaussian RBF kernel function, $\tilde K(x,y) = \exp\left(-\frac{\|x-y\|_2^2}{\mu}\right)$. Letting $\mathcal{Y}$ be the Sobolev space $H^1$, it can be seen that $\phi_g[\theta] \in \mathcal{Y}$ for all $\theta \in C^{2}\left([0,T],\mathbb{R}^n\right)$. This follows since $g$ is infinitely differentiable and $\phi_g[\theta](t) = g(\theta(t))$ must then be twice continuously differentiable.
\end{example}

\begin{corollary}\label{cor:GaussianSobolev}Example \ref{ex:GaussianSobolev} provides a vvRKHS of functions from $C^{2}\left([0,T],\mathbb{R}^n\right)$ to $H^1$, and it is a signal valued RKHS.\end{corollary}

\begin{proof}
Let $\theta \in C^2([0,T],\mathbb{R}^n)$ and $g \in \tilde H$. The norm on $\mathcal{Y}$, which consists of absolutely continuous functions over $[0,T]$, is given as $\| v \|_{H^1}^2 = \int_0^T |v(t)|^2 + |v'(t)|^2 dt$ (cf. \cite{rosenfeld2015densely}).

Consider $\| g\circ \theta \|_{H^1}^2 = \int_0^T |g(\theta(t))|^2 + |\nabla g(\theta(t))\dot \theta(t)|^2 dt.$ Note that for each $t$, $g(\theta(t)) = \langle g, \tilde K_{\theta(t)} \rangle_{\tilde H}$ and {\color{black}$\nabla g(\theta(t))\dot\theta(t)= \langle g, \nabla_2 \tilde K(\cdot,\theta(t))\dot \theta(t) \rangle_{\tilde H}$ (see \cite[p. 132]{steinwart2008support}), where $\nabla_2$ denotes gradient with respect to the second argument and $\tilde{K}_y(x):=\tilde{K}(x,y)$, for $x,y \in \mathbb{R}^n$.} Hence, via Cauchy Schwarz $ \| g\circ \theta \|_{H^1}^2 \le \|g\|_{\tilde H}^2 \int_0^T \| \tilde K_{\theta(t)} \|_{\tilde H}^2 + \color{black}\| \nabla_2 \tilde K(\cdot,\theta(t))\dot \theta(t) \|_{\tilde H}^2 dt.$ Note that $\| \tilde K_{\theta(t)} \|_{\tilde H} = \sqrt{ \tilde K(\theta(t),\theta(t)) }$ and $\tilde K(x,y)$ is the Gaussian kernel. Hence, $\| \tilde K_{\theta(t)} \|_{\tilde H}$ is continuous, and as $\theta$ is continuous and $[0,T]$ is compact, $\int_0^T \| \tilde K_{\theta(t)} \|_{\tilde H}^2 dt$ is bounded. A similar argument reveals that $\int_0^T \color{black}\| \nabla_2 \tilde K(\cdot,\theta(t))\dot \theta(t) \|_{\tilde H}^2 dt$ is bounded. Thus the composition operator $g \mapsto g \circ \theta$ is bounded from $\tilde H$ to $\mathcal{Y}$.\end{proof}

DMD relies on the action of Liouville operators on occupation kernels. As such, it is necessary to define second order occupation kernels in the context of vvRKHSs of the form in Theorem \ref{thm:construction}. 
To motivate the definition of higher order occupation kernels, recall Cauchy's formula for iterated integrals given as $h^{(-m)}(T) = \frac{1}{(m-1)!} \int_0^T (T-t)^{m-1} h(t) dt,$ {\color{black}where $h^{(-m)}(t):= \int_0^t \int_0^{\tau_1}\cdots\int_0^{\tau_{m-1}} h(\tau_m)d\tau_m\cdots d\tau_2 d\tau_1$.}
For a RKHS of continuous functions, $\mathcal{Y}$, as given in Theorem \ref{thm:construction}, the mapping $h \mapsto h^{(-m)}(T)$ is a bounded functional. As a result, by the Reisz representation theorem, there exists $1^{(-m)} \in \mathcal{Y}$ such that $\langle h, 1^{(-m)} \rangle_{\mathcal{Y}} = \frac{1}{(m-1)!} \int_0^T (T-t)^{m-1} h(t) dt$. Note that for $ g \in \tilde{H} $, $\theta \in C^{d}\left([0,T],\mathbb{R}^n\right)$, and $\psi \in H$ such that $\psi = \mathcal{T}g$, the functional
\[
    \psi \mapsto \langle g \circ \theta, 1^{(-m)} \rangle_{\mathcal{Y}}  = \frac{1}{(m-1)!} \int_0^T (T-t)^{m-1} g(\theta(t)) dt
\]
is bounded. As such, by the Reisz representation theorem, there exists $ \Gamma^{(m)}_{\theta} \in H $ such that $\langle \psi, \Gamma^{(m)}_{\theta}\rangle_{H} = \langle g \circ \theta, 1^{(-m)} \rangle_{\mathcal{Y}}$. We define $ \Gamma^{(m)}_{\theta} $ as the \emph{$m$-th order occupation kernel corresponding to $\theta \in C^{d}\left([0,T],\mathbb{R}^n\right)$ in $H$}.

Due to the fact that
$\langle g \circ \theta, 1^{(-m)} \rangle_{\mathcal{Y}} = \langle \psi[\theta], 1^{(-m)} \rangle_{\mathcal{Y}} = \langle \psi, K_{\theta,1^{(-m)}}\rangle_{H},$
the $m$-th order occupation kernel corresponding to $\theta$ can be identified with the kernel function $K_{\theta,1^{(-m)}} \in H$ of the vvRKHS. 

Thus, in contrast with \cite{rosenfeld2019occupation}, where occupation kernels are integrals of the kernel function of an RKHS along trajectories, the $m$-th order occupation kernels defined here are a subset of the set of vector valued kernels in a vvRKHS.

\begin{example}As an example application of signal valued RKHSs, consider, for a function $f:\mathbb{R}^n \to \mathbb{R}^n$, the action of the Liouville operator on $\psi \in H$, given as $A_f\psi[\theta](t) = \nabla \psi[\theta](t)f[\theta](t)$, where $f[\theta](t) := f(\theta(t)),$ and $\nabla \psi[\theta] := \nabla g(\theta(\cdot))$ where $\mathcal{T}g = \psi$. Assume that $A_f$ is densely defined over $H$. If $\gamma:[0,T]\to\mathbb{R}^n$ satisfies $\dot \gamma = f(\gamma)$ it follows that 
\begin{gather*}\langle A_f \psi, \Gamma^{(1)}_\gamma \rangle_{H} = \int_0^T \nabla \psi[\gamma](t) f[\gamma](t) = \int_0^T \nabla g(\gamma(t)) f(\gamma(t)) dt = \int_0^T \frac{d}{dt} g(\gamma(t)) dt =\\ g(\gamma(T)) - g(\gamma(0)) = \psi[\gamma](T) - \psi[\gamma](0) = \langle \psi, K_{\gamma,\mathscr{K}_T} - K_{\gamma,\mathscr{K}_0} \rangle_H,\end{gather*}
{\color{black}where, $\mathscr{K}_t(s):=\mathscr{K}(s,t)$ for $t,s\in[0,T]$.} Hence, $A_f^* \Gamma^{(1)}_\gamma =  K_{\gamma,\mathscr{K}_T} - K_{\gamma,\mathscr{K}_0},$ which coincides with the relation between occupation kernels and Liouville operators over RKHSs \cite{rosenfeld2019occupation}. 
\end{example}

\section{Higher Order Liouville Operators and Occupation Kernels}\label{sec:hoLiouville}

The structure of Liouville operators, given formally as $A_{f} g(x) = \nabla g(x) f(x)$, derive their form from the orbital derivative. In particular, suppose that $\gamma:[0,T] \to \mathbb{R}^n$ satisfies $\dot \gamma = f(\gamma)$, then $A_f g(\gamma(t)) = \nabla g(\gamma(t)) \dot\gamma(t) = \frac{d}{dt} g(\gamma(t)).$ Consequently, higher order Liouville operators may be derived via the same process, where $g$ is composed with $\gamma$ and higher order derivatives with respect to time are taken. To wit, letting $\mathcal{H}[g]$ denote the Hessian of $g:\mathbb{R}^n\to \mathbb{R}$, 
\[ \frac{d^2}{dt^2} g \circ \gamma (t) = \dot \gamma(t)^T \mathcal{H}[g](\gamma(t)) \dot \gamma(t) + \nabla g(\gamma(t)) \ddot \gamma(t).\]
In the case that $\gamma$ satisfies $\ddot \gamma = f(\gamma)$, this becomes
{
    \medmuskip=0mu
    \thinmuskip=0mu
    \thickmuskip=0mu
    \begin{gather*} \frac{d^2}{dt^2} g \circ \gamma (t) = \nabla g(\gamma(t)) f(\gamma(t)) + \left(\dot\gamma(0) + \int\limits_0^t f(\gamma(\tau)) d\tau\right)^T \mathcal{H}[g](\gamma(t)) \left(\dot\gamma(0) + \int\limits_0^t f(\gamma(\tau)) d\tau\right).
    \end{gather*}
}

Fixing $H$ as a signal valued RKHS as in Theorem \ref{thm:construction}, let $f:\mathbb{R}^n \to \mathbb{R}^n$ be the symbol for a second order Liouville operator, $B_f : \mathcal{D}(B_f) \to H$, defined as
{
    \medmuskip=0mu
    \thinmuskip=0mu
    \thickmuskip=0mu
    \begin{gather}
        B_f \psi[\theta](t) := \nabla \psi[\theta](t) f(\theta(t))  + 
        \left(\dot\theta(0) + \int_0^t f(\theta(\tau)) d\tau\right)^T \mathcal{H}\psi[\theta](t) \left(\dot\theta(0) + \int_0^t f(\theta(\tau)) d\tau\right) \label{eq:second-order-liouville} ,
    \end{gather}
}
where $\mathcal{D}(B_f)$ is precisely the collection of $\psi$ for which $B_f\psi \in H$. Note that since $\psi = \phi_g$ for some $g \in \tilde H$, $\frac{\partial}{\partial x_i} \psi[\theta](t)$ is defined as $\frac{\partial}{\partial x_i} g(\theta(t))$ for $i = 1,\ldots,n$, which facilitates the definitions of the gradient and Hessian of $\psi$. Owing to the integral appearing in \eqref{eq:second-order-liouville}, the operator $B_f$ needs to be posed over a Hilbert space consisting of functions of trajectories. Additionally, in contrast to the first order Liouville operator, $B_f$ is linear in $\psi$ but not in the symbol, $f$. 

The operator $B_f$ is connected to second order occupation kernels in the following manner. If $\ddot \gamma = f(\gamma)$, then $\langle B_f \psi, \Gamma^{(2)}_\gamma \rangle_H=\int_0^T (T-t) B_f \psi[\gamma](t) dt =
\int_0^T (T-t) \ddot \psi[\gamma](t) dt =\\  \psi[\gamma](T) - \psi[\gamma](0) - T \nabla \psi[\gamma](0)\dot\gamma(0) = \langle \psi, K_{\gamma,\mathscr{K}_T} - K_{\gamma,\mathscr{K}_0} - T K_{\gamma,\mathscr{K}'_0}\rangle_H,$ {\color{black}where $ \mathscr{K}'_0 := s\mapsto \left.\nicefrac{d(\mathscr{K}(s,t))}{dt}\right\vert_{t=0} \in \mathcal{Y} $.}

Hence, the functional $\psi \mapsto \langle B_f \psi, \Gamma^{(2)}_\gamma \rangle_H$ is bounded, and the following proposition is established.
\begin{proposition}\label{prop:gammainadjoint}Let $f$ be the symbol for a densely defined\footnote{An operator $B_f:\mathcal{D}(B_f)\to H $ is called densely defined if $\mathcal{D}(B_f)$ is a dense subset of $H$.} second order Liouville operator, $B_f$, over a signal valued RKHS and $\gamma \in C^{2}\left([0,T],\mathbb{R}^n\right)$ be such that $\ddot \gamma = f(\gamma)$. Then, $\Gamma^{(2)}_\gamma \in \mathcal{D}(B_f^*)$, and
$B_f^* \Gamma_\gamma^{(2)} = K_{\gamma,\mathscr{K}_T} - K_{\gamma,\mathscr{K}_0} - T K_{\gamma,\mathscr{K}'_0}.$
\end{proposition}
\section{Dynamic Mode Decompositions for Second Order Dynamical Systems}\label{sec:dmdsecondorder}

The objective of this section is to give a data driven model for a state governed by an unknown second order dynamical system. The development follows that of occupation kernel DMD detailed in \cite{rosenfeldDMD}.

The approach is to determine a finite rank representation of $B_f$ over $H$ and to perform an eigendecomposition on this representation to obtain eigenfunctions and eigenvectors for the representation. Following this, the full state observable is decomposed with respect to the eigenfunctions, which ultimately allows for a model to be extracted for the dynamical system.

Suppose that $\varphi \in \mathcal{D}(B_f)$ is an eigenfunction for $B_f$ with eigenvalue $\lambda.$ Then for $\ddot \gamma = f(\gamma)$, $\ddot \varphi[\gamma](t) = B_f \varphi[\gamma](t) = \lambda \varphi[\gamma](t)$. Hence,
\begin{gather*}
\varphi[\gamma](t) = \frac{1}{2} \left( \varphi[\gamma](0) + \frac{\nabla \varphi[\gamma](0) \dot \gamma(0)}{\sqrt{\lambda}} \right) e^{ \sqrt{\lambda} t} + \frac12 \left( \varphi[\gamma](0) - \frac{\nabla \varphi[\gamma](0) \dot \gamma(0)}{\sqrt{\lambda}} \right) e^{- \sqrt{\lambda} t}.\end{gather*}
The full state observable for the signal valued case is then given as $\psi_{id}[\theta] = \theta$. The objective is to decompose each dimension of the full state observable with respect to an eigenbasis, $\{ \varphi_i \}_{i=1}^\infty$ with eigenvalues $\{ \lambda_i \}_{i=1}^\infty$, of $B_f$, provided that one exists, so as to express $\psi_{id}[\gamma](t)$ as
{\color{black}\begin{gather}
    \psi_{id}[\gamma](t) = \gamma(t) = \lim_{M\to\infty}\sum_{m=1}^M \xi_{m,M} \left(\frac{1}{2} \left( \varphi_m[\gamma](0) + \frac{\nabla \varphi_m[\gamma](0) \dot \gamma(0)}{\sqrt{\lambda_m}} \right) e^{ \sqrt{\lambda_m} t}\right.\nonumber\\
+ \left.\frac12 \left( \varphi_m[\gamma](0) - \frac{\nabla \varphi_m[\gamma](0) \dot \gamma(0)}{\sqrt{\lambda_m}} \right) e^{- \sqrt{\lambda_m} t}\right).\label{eq:dmd}
\end{gather}
Since the eigenbasis may not be orthogonal, addition of each new eigenfunction to the linear combination in \eqref{eq:dmd} may affect the coefficients corresponding to all other eigenfunctions. This dependence of the coefficients on the collection of basis functions is expressed through the second subscript of $M$. In the following, finite-rank representations of the coefficients $\xi_{m,M}$ are referred to as the second order Liouville modes for the dynamical system.}

Since $B_f$ is not known when $f$ is unknown. A finite rank proxy of $B_f$ needs to be constructed from the observed trajectories. In the place of the eigenfunctions of $B_f$, the eigenfunctions of a finite rank representation will be leveraged to determine an estimate for \eqref{eq:dmd}. Let $\{ \gamma_i \}_{i=1}^M \subset C^{2}\left([0,T],\mathbb{R}^n\right)$ be a collection of observed trajectories for the second order dynamical system, and let $\alpha = \{\Gamma^{(2)}_{\gamma_i}\}_{i=1}^M$ be the corresponding collection of second order occupation kernels in $H$. Let $P_{\alpha}$ be the projection onto $\vspan \alpha$. 

{\color{black} 
\begin{assumption}
    The derivation of the Liouville modes relies the assumptions that, \begin{enumerate*}[label = (\arabic*)]
        \item the second order Liouville operator $B_f$ admits eigenfunctions,
        \item the domain of $B_f$ contains the occupation kernels, and 
        \item the closed span of the eigenfunctions of $B_f$ contains the full state observable $\psi_{id}$.
    \end{enumerate*}
\end{assumption}
For $h\in \mathcal{D}(B_f)$, the coefficients $\{b_i\}_{i=1}^M$ in the projection of $B_f h$ onto $\vspan \alpha$, given by $P_\alpha B_f h = \sum_{i=1}^M b_i \Gamma_{\gamma_i}^{(2)}$, can be expressed as
\begin{equation*}
    \begin{pmatrix}
        b_1\\\vdots\\b_M
    \end{pmatrix} = \begin{pmatrix}\langle \Gamma^{(2)}_{\gamma_1}, \Gamma^{(2)}_{\gamma_1} \rangle_H & \cdots &  \langle \Gamma^{(2)}_{\gamma_1}, \Gamma^{(2)}_{\gamma_M} \rangle_H\\
\vdots & \ddots & \vdots\\
\langle \Gamma^{(2)}_{\gamma_M}, \Gamma^{(2)}_{\gamma_1} \rangle_H & \cdots &  \langle \Gamma^{(2)}_{\gamma_M}, \Gamma^{(2)}_{\gamma_M} \rangle_H \end{pmatrix}^{-1} \begin{pmatrix}
    \left\langle B_f h,\Gamma_{\gamma_1}^{(2)}\right\rangle_H\\\vdots\\\left\langle B_f h,\Gamma_{\gamma_M}^{(2)}\right\rangle_H
\end{pmatrix}.
\end{equation*}
Assuming that the occupation kernels are in the domain of the Liouville operator, i.e., $\alpha \subset \mathcal{D}(B_f)$, for $h\in \vspan \alpha$, given by $ h = \sum_{i=1}^M a_i \Gamma_{\gamma_i}^{(2)}$, we have
\begin{multline*}
    \left\langle B_f h,\Gamma_{\gamma_j}^{(2)}\right\rangle_H = \sum_{i=1}^M a_i\left\langle B_f\Gamma_{\gamma_i}^{(2)},\Gamma_{\gamma_j}^{(2)}\right\rangle_H = \sum_{i=1}^M a_i\left\langle \Gamma_{\gamma_i}, B_f^* \Gamma_{\gamma_j}^{(2)}\right\rangle_H \\= \left(\left\langle \Gamma_{\gamma_1}^{(2)}, B_f^* \Gamma_{\gamma_j}^{(2)}\right\rangle_H,\ldots,\left\langle \Gamma_{\gamma_M}^{(2)}, B_f^* \Gamma_{\gamma_j}^{(2)}\right\rangle_H\right)\begin{pmatrix}
        a_1\\\vdots\\a_M
    \end{pmatrix}.
\end{multline*}
As a result, a finite rank representation of $B_f$ restricted to $\vspan \alpha$, i.e., the matrix $[P_\alpha B_f]_\alpha^\alpha$ that maps the coefficients $\{a_i\}_{i=1}^M$ to the coefficients $\{b_i\}_{i=1}^M$, is given as
{\thickmuskip = 0mu\medmuskip=0mu\thinmuskip=0mu
\begin{gather}\small
    [P_\alpha B_f]_{\alpha}^\alpha =
    \begin{pmatrix}\langle \Gamma^{(2)}_{\gamma_1}, \Gamma^{(2)}_{\gamma_1} \rangle_H & \cdots &  \langle \Gamma^{(2)}_{\gamma_1}, \Gamma^{(2)}_{\gamma_M} \rangle_H\\
    \vdots & \ddots & \vdots\\
    \langle \Gamma^{(2)}_{\gamma_M}, \Gamma^{(2)}_{\gamma_1} \rangle_H & \cdots &  \langle \Gamma^{(2)}_{\gamma_M}, \Gamma^{(2)}_{\gamma_M} \rangle_H \end{pmatrix}^{-1}
    \begin{pmatrix}\langle B_f^* \Gamma^{(2)}_{\gamma_1}, \Gamma^{(2)}_{\gamma_1} \rangle_H & \cdots &  \langle B_f^* \Gamma^{(2)}_{\gamma_1}, \Gamma^{(2)}_{\gamma_M} \rangle_H\\
    \vdots & \ddots & \vdots\\
    \langle B_f^* \Gamma^{(2)}_{\gamma_M}, \Gamma^{(2)}_{\gamma_1} \rangle_H & \cdots &  \langle B_f^* \Gamma^{(2)}_{\gamma_M}, \Gamma^{(2)}_{\gamma_M}\rangle_H \end{pmatrix},\label{eq:finiteRank}
\end{gather}}
where $B_f^* \Gamma^{(2)}_{\gamma}$ was given in Proposition \ref{prop:gammainadjoint}. Letting $G$ denote the Gram matrix $(\langle \Gamma^{(2)}_{\gamma_i}, \Gamma^{(2)}_{\gamma_j} \rangle_H)_{i,j=1}^M$, a normalized ``eigenfunction'' can be extracted from an eigenvector, $\nu_j$, of $[P_\alpha B_f]_\alpha^\alpha$ with eigenvalue $\lambda_j$ as 
\begin{equation}
    \hat \varphi_j = \frac{1}{\sqrt{\nu_j^T G \nu_j}} \sum_{i=1}^M (\nu_j)_i \Gamma^{(2)}_{\gamma_i},\label{eq:eigenfunction}
\end{equation}
which can be leveraged as a proxy for a proper eigenfunction of $B_f$, in keeping with the implementation of DMD for Koopman and Liouville operators. In \eqref{eq:eigenfunction} and in the following development, $(x)_i$ denotes the projection onto the $i$-th coordinate of $x \in \mathbb{R}^n$.

The second order Liouville modes can then be constructed by examining the inner products $\left\langle(\psi_{id})_i,\Gamma_{\gamma_j}^{(2)}\right\rangle_H$, where $(\psi_{id})_i$ is the $i$-th component of the full state observable, i.e., $(\psi_{id})_i[\theta](t) := (\theta(t))_i$. The second order Liouville modes $ \{(\xi_m)_i\}_{m=1}^M $ are defined as the coefficients in the projection of $(\psi_{id})_i$ onto the span of the normalized eigenfunctions in \eqref{eq:eigenfunction} that is, 
\begin{gather*}\small
    \begin{pmatrix}
        \left\langle(\psi_{id})_1,\Gamma_{\gamma_j}^{(2)}\right\rangle_H\\\vdots\\\left\langle(\psi_{id})_n,\Gamma_{\gamma_j}^{(2)}\right\rangle_H
    \end{pmatrix} \approx \begin{pmatrix}
        \left\langle\sum_{m=1}^M (\xi_m)_1\hat\varphi_m,\Gamma_{\gamma_j}^{(2)}\right\rangle_H\\\vdots\\\left\langle\sum_{m=1}^M (\xi_m)_n\hat\varphi_m,\Gamma_{\gamma_j}^{(2)}\right\rangle_H
    \end{pmatrix} = \begin{pmatrix}
        \left\langle\sum_{m=1}^M (\xi_m)_1\sum_{k=1}^M\frac{(\nu_m)_k\Gamma_{\gamma_k}^{(2)}}{\sqrt{\nu_m^T G \nu_m}},\Gamma_{\gamma_j}^{(2)}\right\rangle_H\\\vdots\\\left\langle\sum_{m=1}^M (\xi_m)_n\sum_{k=1}^M\frac{(\nu_m)_k\Gamma_{\gamma_k}^{(2)}}{\sqrt{\nu_m^T G \nu_m}},\Gamma_{\gamma_j}^{(2)}\right\rangle_H
    \end{pmatrix} \\
    =\sum_{m=1}^M \xi_m \sum_{k=1}^M \frac{(\nu_m)_k}{\sqrt{\nu_m^T G \nu_m}} \left\langle\Gamma_{\gamma_k}^{(2)}, \Gamma_{\gamma_j}^{(2)}\right\rangle_H = \sum_{m=1}^M  \frac{\xi_m \nu_m^T G^j}{\sqrt{\nu_m^T G \nu_m}},
\end{gather*}
where $G^j$ denotes the $j$-th column of the Gram matrix. The matrix $\xi := \begin{pmatrix} \xi_1&\cdots&\xi_M \end{pmatrix}$ of second order Liouville modes is then given by
\begin{gather}
    \xi = \begin{pmatrix}
        \left\langle(\psi_{id})_1,\Gamma_{\gamma_1}^{(2)}\right\rangle_H&\cdots&\left\langle(\psi_{id})_1,\Gamma_{\gamma_M}^{(2)}\right\rangle_H\\\vdots&\ddots&\vdots\\\left\langle(\psi_{id})_n,\Gamma_{\gamma_1}^{(2)}\right\rangle_H&\cdots&\left\langle(\psi_{id})_n,\Gamma_{\gamma_M}^{(2)}\right\rangle_H
    \end{pmatrix}\left(
    \begin{pmatrix}
        \frac{\nu_1^T}{\sqrt{\nu_1^T G \nu_1}}\\\vdots\\\frac{\nu_M^T}{\sqrt{\nu_M^T G \nu_M}}
    \end{pmatrix}G\right)^{-1}. \label{eq:LiouvilleModes}
\end{gather}
}

\section{Evaluation of the Occupation Kernels}\label{sec:evaluation}

{\color{black}Equations \eqref{eq:finiteRank}, \eqref{eq:eigenfunction}, and \eqref{eq:LiouvilleModes} constitute the algorithm for obtaining the second order dynamic modes of a dynamical system from observed trajectories, which can be utilized via \eqref{eq:dmd} to obtain a data driven model for the dynamical system that reconstructs trajectories of the system starting from given initial conditions. To implement the algorithm, one needs to evaluate} the following quantities:
$\langle \Gamma_{\gamma_i}^{(2)}, \Gamma_{\gamma_j}^{(2)} \rangle_H$,
$\langle B_f^* \Gamma_{\gamma_i}^{(2)}, \Gamma_{\gamma_j}^{(2)} \rangle_H$,
$\Gamma_{\gamma_i}^{(2)}[\theta]$, {\color{black}and $\left\langle(\psi_{id})_k,\Gamma_{\gamma_i}^{(2)}\right\rangle_H$,
where $i,j=1,\ldots,M$, $k=1,\ldots,n$,} and $\theta \in C^{2}\left([0,T],\mathbb{R}^n\right)$.

First note that as a surjective isometry, the operator $\mathcal{T}$, given in Theorem \ref{thm:construction}, is an unitary operator. Moreover, if $\psi \in H$ such that $g = \mathcal{T}^{-1}\psi$, then for $\theta \in C^{2}\left([0,T],\mathbb{R}^n\right)$ and $t \in [0,T]$, the following holds
$\langle \psi, K_{\theta,\mathscr{K}_t} \rangle_H = \psi[\theta](t) = g(\theta(t)) = \langle g, K_{\theta(t)} \rangle_{\tilde H}.$ Since $\psi$ was arbitrary, it follows that $K_{\theta,\mathscr{K}_t} = \mathcal{T} \tilde K_{\theta(t)}.$ Thus,
\[
    \langle K_{\theta_1,\mathscr{K}_{t_1}}, K_{\theta_2,\mathscr{K}_{t_2}} \rangle_H = \langle \mathcal{T}\tilde K_{\theta_1(t_1)}, \mathcal{T}\tilde K_{\theta_2(t_2)} \rangle_{ H}=\langle \tilde K_{\theta_1(t_1)},\tilde K_{\theta_2(t_2)} \rangle_{\tilde H} = \tilde K(\theta_2(t_2),\theta_1(t_1)).
\]
As a result, the following holds
\begin{gather*}\Gamma_{\gamma_i}^{(2)}[\theta](t) = \langle \Gamma_{\gamma_i}^{(2)}, K_{\theta,\mathscr{K}_t} \rangle_H=
\int_0^T (T-\tau) K_{\theta,\mathscr{K}_t}[\gamma_i](\tau) d\tau = 
\int_0^T (T-\tau) \tilde K(\theta(t),\gamma_i(\tau)) d\tau.\end{gather*}
Consequently,
\begin{gather*}
   \langle \Gamma_{\gamma_i}^{(2)}, \Gamma_{\gamma_j}^{(2)} \rangle_H = \int_0^T \int_0^T (T-\tau)(T-t) \tilde K(\gamma_i(t),\gamma_j(\tau)) dt d\tau. 
\end{gather*}

To compute $\langle B_f^* \Gamma_{\gamma_i}^{(2)}, \Gamma_{\gamma_j}^{(2)} \rangle_H$. Recall that when $\gamma$ is trajectory satisfying $\ddot \gamma = f(\gamma)$, $B_f^* \Gamma^{(2)}_\gamma = K_{\gamma,\mathscr{K}_T} - K_{\gamma,\mathscr{K}_0} - T K_{\gamma,\mathscr{K}'_0}$ according to Proposition \ref{prop:gammainadjoint}. By linearity, the only unresolved quantity is $\langle K_{\gamma_i,\mathscr{K}'_0}, \Gamma_{\gamma_j}^{(2)} \rangle_H.$ As above consider $\langle \psi, K_{\gamma_i,\mathscr{K}'_0} \rangle_H = \psi'[\gamma_i](0) = g'(\gamma_i(0)) = \langle g, K'_{\gamma_i(0)} \rangle_{\tilde H}.$ Hence,
\[ \langle K_{\gamma_i,\mathscr{K}'_0}, \Gamma_{\gamma_j}^{(2)} \rangle_H = \int_0^T (T-t) \nabla_2 \tilde K(\gamma_j(t),\gamma_i(0))\dot\gamma(0) dt. \]

{\color{black}The remaining computation for $\left\langle(\psi_{id})_k,\Gamma_{\gamma_i}^{(2)}\right\rangle_H$ can be accomplished by recalling that 
\[
    \left\langle(\psi_{id})_k,\Gamma_{\gamma_i}^{(2)}\right\rangle_H =  \left\langle(\psi_{id})_k[\gamma_i],1^{(-2)}\right\rangle_\mathcal{Y} = \left\langle(\gamma_i)_k,1^{(-2)}\right\rangle_\mathcal{Y} = \int_0^T (T-\tau) (\gamma_i(\tau))_k d\tau.
\]

\section{Numerical experiments}\label{sec:experiments}
The following numerical experiments demonstrate the efficacy of the developed DMD method.

\textbf{Experiment 1:} In this experiment, 4 trajectories of the undamped linear oscillator $\ddot{\gamma} = -2\gamma$, over the time horizon $[0,5]$, starting from different initial conditions, are utilized to build a data-driven model in the form of Liouville modes and eigenfunctions and eigenvalues of the finite rank representation of the second order Liouville operator. The trajectories are sampled every 0.5s. To simulate measurement noise, a random number drawn from the zero mean Gaussian distribution with standard deviation 0.01 is added to every sample of the trajectories and the time derivatives at the initial time. 

The Liouville modes and the finite rank eigendecomposition are then used in \eqref{eq:dmd} to reconstruct trajectories $\hat \gamma$ of the system, over a longer time horizon of $[0,10]$, starting from 1000 randomly selected initial conditions in the domain $[-1,1]\times [-1,1]$. The root mean square (RMS) value of the relative error signal $t\mapsto \nicefrac{\gamma(t) - \hat{\gamma}(t)}{\left\vert\gamma\right\vert_\infty}$ is used to quantify the quality of the data-driven model. Figure \ref{fig:linearSystemRMS} shows the relative RMS errors in each of the 1000 trials.
\begin{figure}
    \centering
        \begin{minipage}[t]{0.49\textwidth}
        \includegraphics[width=\textwidth]{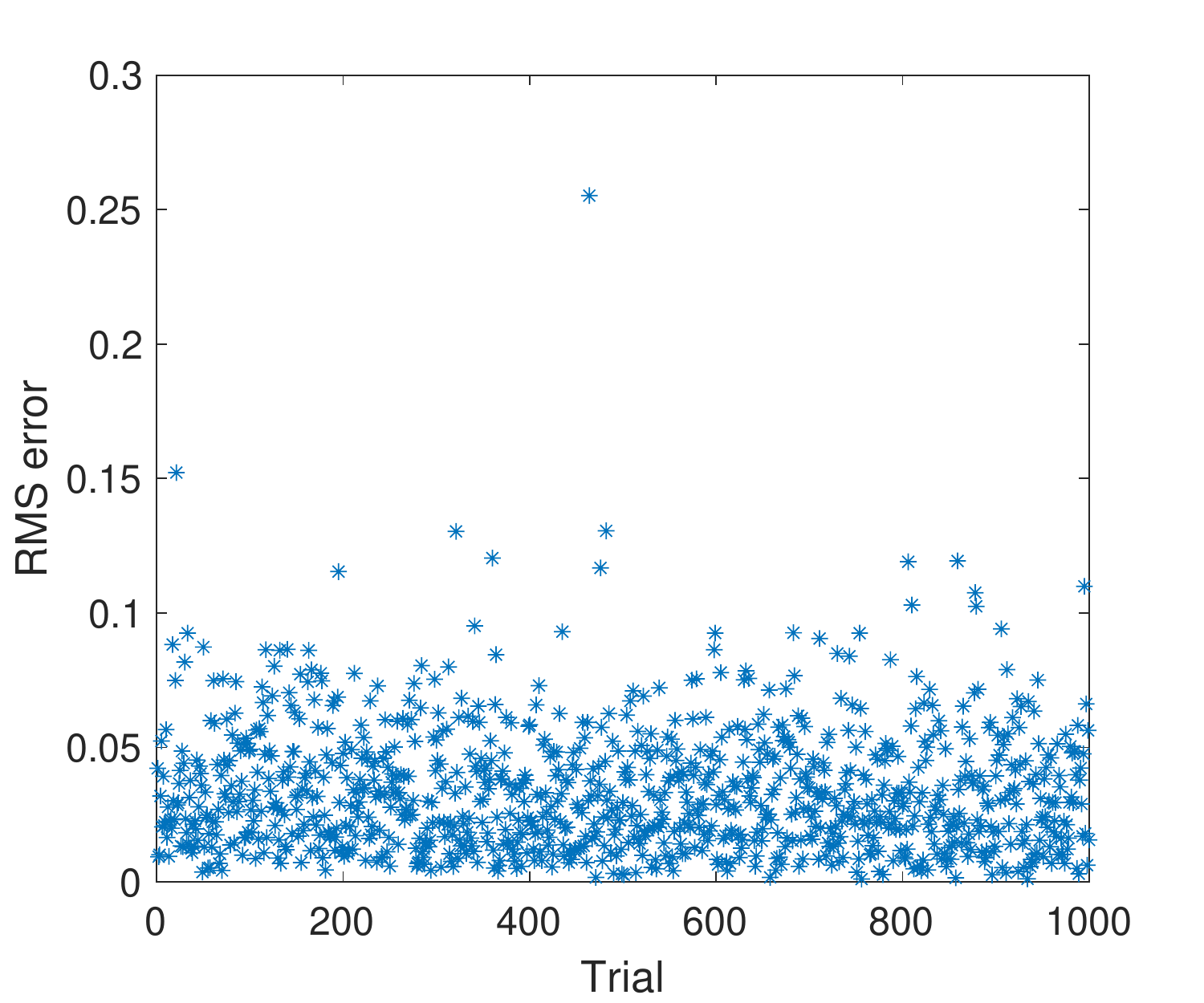}
        \caption{Reconstruction results for the linear oscillator. The root mean square (RMS) value of the relative error signal $t\mapsto \nicefrac{\gamma(t) - \hat{\gamma}(t)}{\left\vert\gamma\right\vert_\infty}$ plotted against trial number, where each trial reconstructs a trajectory of the linear oscillator, over $t\in[0,10]$, starting from a randomly selected initial condition.}
        \label{fig:linearSystemRMS}
    \end{minipage}\hfill\begin{minipage}[t]{0.49\textwidth}
        \includegraphics[width=\textwidth]{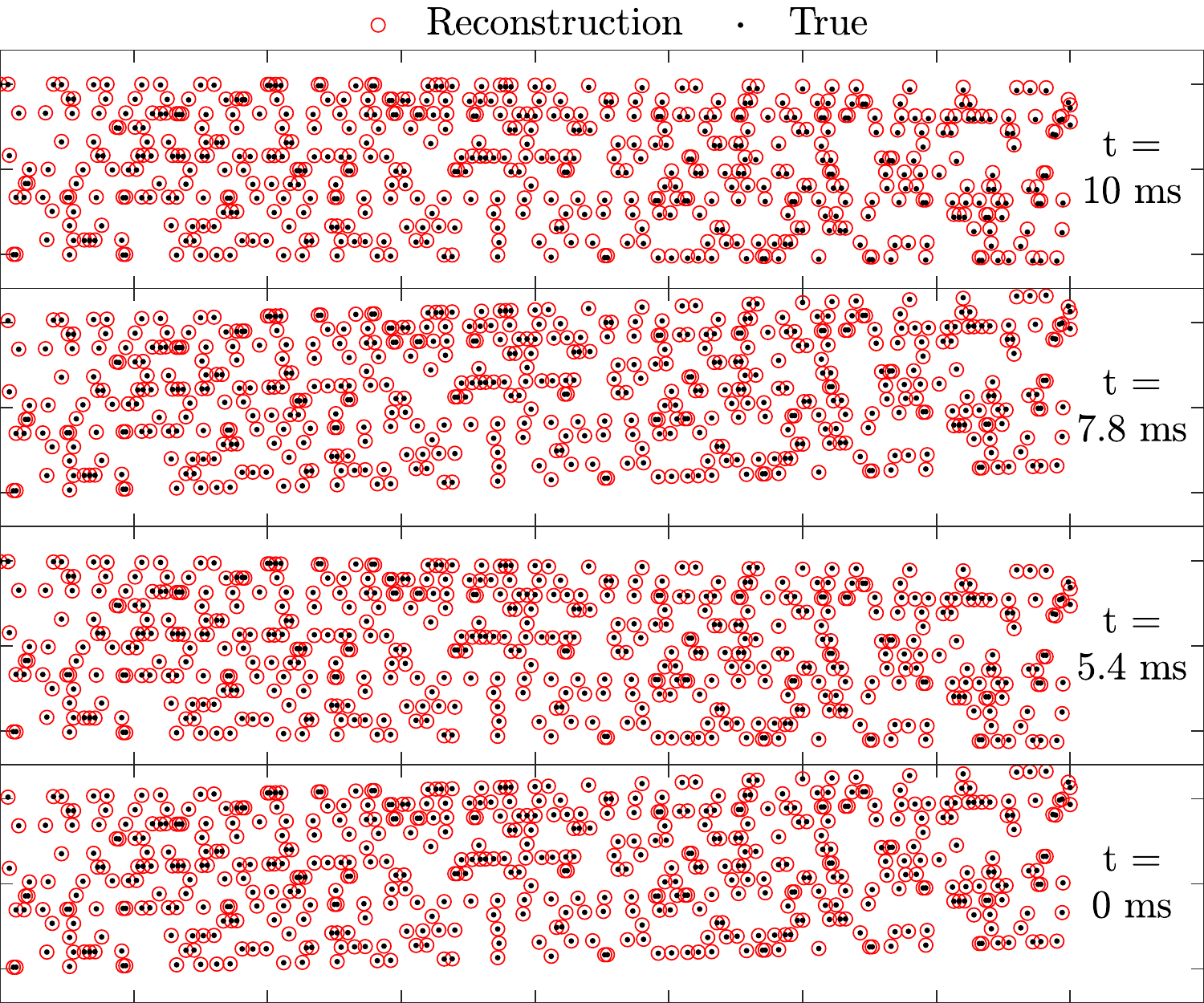}
        \caption{An undamped two dimensional cantilever beam fixed at the left end. True positions (black dots) of select finite element nodes are compared against their reconstructed positions (red circles), computed using a DMD model trained from trajectory segments 2.3 ms long. Reconstruction error is indicated by displacement of the black dots away from the center of the red circles.}
        \label{fig:BeamSnapshots}
    \end{minipage}
\end{figure}

\textbf{Experiment 2:} This experiment comprises of a two dimensional undamped cantilever beam that is initially assumed to be bent under a load applied at the end. The load is removed at the initial time, resulting in sustained oscillations. A single trajectory of the beam, comprised of displacements in the $x$ and $y$ direction over the time horizon $[0,0.0236]$ of 6511 nodes is computed using the MATLAB\textsuperscript{\textregistered} Partial Differential Equation Toolbox \cite{SCC.MATLABPDE2021}. The trajectory is sampled every 0.0788 ms. The resulting dataset of 301 snapshots, each with dimension 13022, is segmented to yield 271 trajectories, 2.3 ms (31 snapshots) long. The 271 trajectories are then utilized for second order dynamic mode decomposition.

The Liouville modes and the finite rank eigendecomposition are then used in \eqref{eq:dmd} to reconstruct the original trajectory $\hat \gamma$ over the time horizon $[0,0.0102]$, starting from the same initial condition as the original trajectory. Figure \ref{fig:BeamSnapshots} shows true and reconstructed positions of select nodes in the finite element analysis (FEA).}
\section{Discussion}

This manuscript presents a framework for second order DMD methods by introducing signal valued RKHSs. These Hilbert spaces may be constructed from traditional scalar valued RKHSs via the construction given in Theorem \ref{thm:construction}. The mapping $T: \tilde H \to H$ is a unitary operator and $H$ can be seen to be unitarily equivalent to $\tilde H$. However, this extra dressing on top of $\tilde H$ that enables the consideration of a signal as the fundamental input is necessary for the definition of higher order Liouville operators, such as $B_f$. Significantly, as can be observed in Section \ref{sec:evaluation}, the range space $\mathcal{Y}$ is not accessed for any of the evaluations. This means that $\mathcal{Y}$ is necessary \emph{only} for theoretical completeness and not for practical implementation. While only second order Liouville operators were presented, signal valued RKHSs are sufficient for the analysis of higher order Liouville operators, with the adjustment of the signal space from $C^{2}\left([0,T],\mathbb{R}^n\right)$ to $C^{d}\left([0,T],\mathbb{R}^n\right)$ for the $d$-th order Liouville operator.

Section \ref{sec:dmdsecondorder} gives an algorithm for the direct DMD analysis of second order systems that avoids augmenting the system state and numerically estimating the time derivative of the system state. Accounting for the computations in Section \ref{sec:evaluation}, it can be seen that the implementation of the algorithm differs only slightly from that of the DMD method in \cite{rosenfeldDMD}. However, it can be observed that it is necessary to numerically estimate the time derivative at a single point in the time domain, $t=0$, which is a dramatic improvement over numerically differentiating the entire trajectory. However, higher order Liouville operators will require additional numerical derivatives at the origin, which may limit the applicability of the method when initial values of high order system is unknown.

{\color{black}Section \ref{sec:experiments} includes two numerical experiments, one that demonstrates the robustness of second order Liouville DMD to measurement noise and another that demonstrates application to systems with a high dimensional state space. The results of Experiment 1 demonstrate that even with a very small amount of noisy data, the models generated via second-order Liouville DMD accurately reconstruct and predict system behavior starting from given initial conditions. The results of Experiment 2 demonstrate the applicability of second order Liouville DMD to a real-world problem with a very high-dimensional state space. The generated data-driven model is able to accurately characterize the transient dynamics of an undamped two dimensional cantilever beam starting from given initial conditions.

The results demonstrated in Section \ref{sec:experiments} are obtained using Gaussian kernels for Experiment 1 and dot product kernels for Experiment 2, but similar results can also be obtained using, among others, Gaussian kernels and linear and exponential dot product kernels. While agnostic to the type of kernels, a successful implementation of the DMD technique does require careful selection of shape parameters, e.g., width for Gaussian kernels, slope for linear kernels, and decay rate for exponential kernels. The parameters were hand-tuned using trial and error for experiments presented in Section \ref{sec:experiments}, and their values are available in the code submitted with the article.} 

\section{Conclusion}\label{sec:conclusion}

This manuscript presented the necessary theoretical foundations for the definition of higher order Liouville operators. This manuscript proceeded to give a DMD algorithm for modeling a second order dynamical system that leveraged second order Liouville operators and occupation kernels. Through the examination of the components of the algorithm in Section \ref{sec:evaluation}, it was determined that the actual implementation of the second order DMD routine requires only a slight modification of DMD with occupation kernels in \cite{rosenfeldDMD}, and the range space was ultimately unnecessary for computations.

{\color{black}The experiments in Section \ref{sec:experiments} involve undamped oscillators and undamped beams because damped oscillators and damped beams typically cannot be cast into the form $\ddot{\gamma} = f(\gamma)$, but require the form $\ddot{\gamma} = f(\gamma,\dot{\gamma})$. An immediate future direction for this research is in exploring extension to such models where the higher order derivatives of the output depend not just on the output, but also on its derivatives.

While the primary advantage of second order DMD is that it does not require measurement of derivatives of the trajectory, the authors postulate that even when the derivatives are available, the smaller state dimension in second order DMD would lead to similar accuracy but faster execution, especially in systems with a large number of state variables. Confirmation of the above postulate via a detailed  numerical comparison of the performance (in terms of computation time) of models generated using first order DMD with those generated using second order DMD is also a part of future research.}



\begin{ack}
This research was supported by the Air Force Office of Scientific Research (AFOSR) under contract numbers FA9550-20-1-0127 and FA9550-21-1-0134, and the National Science Foundation (NSF) under award 2027976. Any opinions, findings and conclusions or recommendations expressed in this material are those of the author(s) and do not necessarily reflect the views of the sponsoring agencies.
\end{ack}

\small
\bibliographystyle{IEEEtran}
\bibliography{references}

\begin{thebibliography}{10}
\providecommand{\url}[1]{#1}
\csname url@samestyle\endcsname
\providecommand{\newblock}{\relax}
\providecommand{\bibinfo}[2]{#2}
\providecommand{\BIBentrySTDinterwordspacing}{\spaceskip=0pt\relax}
\providecommand{\BIBentryALTinterwordstretchfactor}{4}
\providecommand{\BIBentryALTinterwordspacing}{\spaceskip=\fontdimen2\font plus
\BIBentryALTinterwordstretchfactor\fontdimen3\font minus
  \fontdimen4\font\relax}
\providecommand{\BIBforeignlanguage}[2]{{%
\expandafter\ifx\csname l@#1\endcsname\relax
\typeout{** WARNING: IEEEtran.bst: No hyphenation pattern has been}%
\typeout{** loaded for the language `#1'. Using the pattern for}%
\typeout{** the default language instead.}%
\else
\language=\csname l@#1\endcsname
\fi
#2}}
\providecommand{\BIBdecl}{\relax}
\BIBdecl

\bibitem{budivsic2012applied}
M.~Budi{\v{s}}i{\'c}, R.~Mohr, and I.~Mezi{\'c}, ``Applied {K}oopmanism,''
  \emph{Chaos: An Interdisciplinary Journal of Nonlinear Science}, vol.~22,
  no.~4, p. 047510, 2012.

\bibitem{kutz2016dynamic}
J.~N. Kutz, S.~L. Brunton, B.~W. Brunton, and J.~L. Proctor, \emph{Dynanic mode
  decomposition - data-driven modeling of complex systems}.\hskip 1em plus
  0.5em minus 0.4em\relax Philadelphia, PA: Society for Industrial and Applied
  Mathematics, 2016.

\bibitem{proctor2016dynamic}
J.~L. Proctor, S.~L. Brunton, and J.~N. Kutz, ``Dynamic mode decomposition with
  control,'' \emph{SIAM Journal on Applied Dynamical Systems}, vol.~15, no.~1,
  pp. 142--161, 2016.

\bibitem{mauroy2020linear}
A.~Mauroy and J.~Goncalves, ``Koopman-based lifting techniques for nonlinear
  systems identification,'' \emph{IEEE Transactions on Automatic Control},
  vol.~65, no.~6, pp. 2550--2565, Sep. 2020.

\bibitem{mauroy2016global}
A.~Mauroy and I.~Mezi{\'c}, ``Global stability analysis using the
  eigenfunctions of the {K}oopman operator,'' \emph{IEEE Transactions on
  Automatic Control}, vol.~61, no.~11, pp. 3356--3369, 2016.

\bibitem{brunton2016discovering}
S.~L. Brunton, J.~L. Proctor, and J.~N. Kutz, ``Discovering governing equations
  from data by sparse identification of nonlinear dynamical systems,''
  \emph{Proceedings of the National Academy of Sciences}, vol. 113, no.~15, pp.
  3932--3937, 2016.

\bibitem{williams2015kernel}
M.~O. Williams, C.~W. Rowley, and I.~G. Kevrekidis, ``A kernel-based method for
  data-driven {K}oopman spectral analysis,'' \emph{Journal of Computational
  Dynamics}, vol.~2, no.~2, pp. 247--265, 2015.

\bibitem{haddad2019dynamical}
W.~Haddad, \emph{A Dynamical Systems Theory of Thermodynamics}, ser. Princeton
  Series in Applied Mathematics.\hskip 1em plus 0.5em minus 0.4em\relax
  Princeton University Press, 2019.

\bibitem{toth2018reaction}
J.~T{\'o}th, A.~L. Nagy, and D.~Papp, \emph{Reaction kinetics: exercises,
  programs and theorems}.\hskip 1em plus 0.5em minus 0.4em\relax Springer,
  2018.

\bibitem{hallam2012mathematical}
T.~G. Hallam and S.~A. Levin, \emph{Mathematical ecology: an introduction},
  ser. Biomathematics.\hskip 1em plus 0.5em minus 0.4em\relax Springer Science
  \& Business Media, 2012, vol.~17.

\bibitem{rosenfeld2019occupation}
J.~A. Rosenfeld, B.~Russo, R.~Kamalapurkar, and T.~T. Johnson, ``The occupation
  kernel method for nonlinear system identification,'' arXiv:1909.11792, 2019.

\bibitem{lasserre2008nonlinear}
J.~B. Lasserre, D.~Henrion, C.~Prieur, and E.~Tr{\'e}lat, ``Nonlinear optimal
  control via occupation measures and lmi-relaxations,'' \emph{SIAM Journal on
  Control and Optimization}, vol.~47, no.~4, pp. 1643--1666, 2008.

\bibitem{rosenfeldCDC2019}
J.~A. Rosenfeld, R.~Kamalapurkar, B.~Russo, and T.~T. Johnson, ``Occupation
  kernels and densely defined {L}iouville operators for system
  identification,'' in \emph{Proceedings of the IEEE Conference on Decision and
  Control}, Dec. 2019, pp. 6455--6460.

\bibitem{rosenfeldDMD}
J.~Rosenfeld, R.~Kamalapurkar, L.~F. Gruss, and J.~Taylor, ``Dynamic mode
  decomposition for continuous time systems with the {L}iouville operator,''
  {arXiv:1910.03977}, 2020.

\bibitem{rosenfeld2021dynamic}
J.~A. Rosenfeld and R.~Kamalapurkar, ``Dynamic mode decomposition with control
  {L}iouville operators,'' in \emph{Proceedings of the International Symposium
  on Mathematical Theory of Networks and Systems}, 2021, to appear, see
  {arXiv:2101.02620}.

\bibitem{li2021fractional}
\BIBentryALTinterwordspacing
X.~Li and J.~A. Rosenfeld, ``Fractional order system identification with
  occupation kernel regression,'' \emph{IEEE Control Systems Letters}, to
  appear. [Online]. Available:
  \url{https://ieeexplore.ieee.org/document/9305713}
\BIBentrySTDinterwordspacing

\bibitem{russo2021motion}
B.~P. Russo, R.~Kamalapurkar, D.~Chang, and J.~A. Rosenfeld, ``Motion
  tomography via occupation kernels,'' {arXiv:2101.02677}.

\bibitem{SCC.Carmeli.DeVito.ea2010}
C.~Carmeli, E.~De~Vito, A.~Toigo, and V.~Umanit{\'{a}}., ``Vector valued
  reproducing kernel {H}ilbert spaces and universality,'' \emph{Anal. Appl.},
  vol.~08, no.~01, pp. 19--61, 2010.

\bibitem{pedrick1957theory}
G.~Pedrick, ``Theory of reproducing kernels for hilbert spaces of vector valued
  functions,'' Ph.D. dissertation, University of Kansas, 1957.

\bibitem{rosenfeld2015densely}
J.~A. Rosenfeld, ``Densely defined multiplication on several sobolev spaces of
  a single variable,'' \emph{Complex Analysis and Operator Theory}, vol.~9,
  no.~6, pp. 1303--1309, 2015.

\bibitem{steinwart2008support}
I.~Steinwart and A.~Christmann, \emph{Support vector machines}.\hskip 1em plus
  0.5em minus 0.4em\relax Springer Science \& Business Media, 2008.

\bibitem{SCC.MATLABPDE2021}
``Dynamic analysis of clamped beam,''
  \url{https://www.mathworks.com/help/pde/ug/dynamic-analysis-of-a-clamped-beam.html},
  accessed: 2021-05-27.

\end{thebibliography}

\end{document}